\title{A Hierarchy of Edge-Weight Symmetries in Perfect Matchings}

\documentclass{article}
\usepackage{fullpage}

\usepackage[utf8]{inputenc}
\usepackage{amsmath}
\usepackage{amsthm,amssymb}
\usepackage{nicefrac,amsfonts}
\usepackage{thmtools,mathtools,leftindex,leftidx}
\usepackage{caption,subcaption}
\usepackage{thm-restate}
\usepackage{epsfig,graphicx,graphics,color,tikz,tikz-cd,tcolorbox}
\usepackage[table]{xcolor}
\usepackage{tcolorbox}
\tcbset{
  colback=white,
  colframe=black,
  boxrule=0.5pt,
  left=1mm,
  right=1mm,
  top=1mm,
  bottom=1mm
}
\usepackage{enumerate,enumitem}
\usepackage[sort,nocompress]{cite}
\usepackage{xspace}
\usepackage{bbm}
\usepackage{hyperref}
\usepackage[capitalise]{cleveref}
\usepackage{float}
\hypersetup{
    colorlinks=true,       % false: boxed links; true: colored links
    linkcolor=blue,        % color of internal links (change box color with linkbordercolor)
    citecolor=magenta,         % color of links to bibliography
    filecolor=magenta,     % color of file links
    urlcolor=cyan,         % color of external links
    linktoc=all
}
\usepackage{titling}

\usepackage{enumitem}
\newenvironment{symenum}
 {\enumerate[label=\normalfont{(\noexpand\thisenumsymbol)},align=parleft,labelindent=0pt,itemindent=0pt,labelsep=35pt,leftmargin=*]}
 {\endenumerate}
\newcommand\thisenumsymbol{}
\newcommand\itemsymbol[1]{%
  \renewcommand{\thisenumsymbol}{#1}%
  \item
}

\newlength{\bibitemsep}
\newlength{\bibparskip}
\let\oldthebibliography\thebibliography
\renewcommand\thebibliography[1]{%
  \oldthebibliography{#1}%
  \setlength{\parskip}{\bibitemsep}%
  \setlength{\itemsep}{\bibparskip}%
}

\makeatletter
\renewcommand{\paragraph}{%
  \@startsection{paragraph}{4}%
  {\z@}{1.6ex \@plus 1ex \@minus .2ex}{-0.5em}%
  {\normalfont\normalsize\bfseries}%
}
\makeatother

\theoremstyle{plain}
\newtheorem{thm}{Theorem}[section]
\newtheorem{lem}[thm]{Lemma}
\newtheorem{cor}[thm]{Corollary}

\newtheorem{qu}[thm]{Question}

\theoremstyle{definition}

\newtheorem{rem}[thm]{Remark}

\newcommand*{\claimproofname}{Proof of claim.}

\makeatletter
\newcommand{\leqnomode}{\tagsleft@true}
\newcommand{\reqnomode}{\tagsleft@false}
\makeatother

\makeatletter
 \newcommand{\linkdest}[1]{\Hy@raisedlink{\hypertarget{#1}{}}}
\makeatother

\def\final{0}  % set this to 1 to get a comment-free version
\def\iflong{\iffalse}
\ifnum\final=0  %namely if we allow comments in the output
\newcommand{\kristof}[1]{{\color{red}[{\tiny \textbf{Kristóf:}  #1}]\marginpar{\color{red}*}}}
\newcommand{\yutaro}[1]{{\color{blue}[{\tiny \textbf{Yutaro:}  #1}]\marginpar{\color{blue}*}}}
\newcommand{\viktor}[1]{{\color{magenta}[{\tiny \textbf{Viktor:}  #1}]\marginpar{\color{magenta}*}}}
\else % in this case [final=1] we don't want any comments to show
\newcommand{\kristof}[1]{}
\newcommand{\yutaro}[1]{}
\newcommand{\viktor}[1]{}
\fi

\DeclareMathOperator\im{im}

\DeclareMathOperator\spa{span}

\newcommand{\bR}{\mathbb{R}}

\newcommand{\bZ}{\mathbb{Z}}

\newcommand{\cC}{\mathcal{C}}

\newcommand{\cF}{\mathcal{F}}

\newcommand{\cM}{\mathcal{M}}
\newcommand{\cN}{\mathcal{N}}

\newcommand{\cP}{\mathcal{P}}

\let\Right\bigr
\let\Left\bigl
\makeatletter
\def\bigr#1{\Right#1\@ifnextchar){\!\bigr}{}}
\def\bigl#1{\Left#1\@ifnextchar({\!\bigl}{}}
\makeatother

\thanksmarkseries{alph}
\author{
Kristóf Bérczi\thanks{MTA-ELTE Matroid Optimization Research Group and HUN-REN–ELTE Egerváry Research Group, Department of Operations Research, ELTE Eötvös Loránd University, and HUN-REN Alfréd Rényi Institute of Mathematics, Budapest, Hungary. Email: \texttt{kristof.berczi@ttk.elte.hu}.}
\and
Viktor Csaplár\thanks{Department of Operations Research, ELTE Eötvös Loránd University, Budapest, Hungary. Email: \texttt{viktorcsaplar@gmail.com}.}
\and
Yutaro Yamaguchi\thanks{Department of Information and Physical Sciences, Graduate School of Information Science and Technology, Osaka University, Osaka, Japan. Email: \texttt{yutaro.yamaguchi@ist.osaka-u.ac.jp}.}
}
\date{}

\begin{document}
\maketitle

\thispagestyle{empty}
%%%%%%%%%%%%%%%%%%%%%%%%%%%%%%%%
\begin{abstract} 
Motivated by the exact weight perfect matching problem and recent parameterized algorithms for finding an $\ell$-th smallest perfect matching, we study structural properties of edge-weight symmetries in graphs. Recent work by El Maalouly et al.\ (ESA 2025) showed that excluding all perfect matchings whose weight is at most the $(\ell - 1)$-th smallest possible value in the graph requires fixing at most $2(\ell-1)$ edges in non-bipartite graphs and at most $\ell-1$ edges in bipartite graphs. A natural open question is whether fixing a single edge is always sufficient to shift the extreme (minimum or maximum) weight of a perfect matching when the global minimum and maximum weights differ.

To address this, we define and analyze a hierarchy of progressively weaker edge-weight properties: \textit{node-induced weights}, \textit{even walk} and \textit{cycle symmetries}, \textit{perfect matching equality}, and the \textit{edge min-max property}. We derive a basic hierarchy among these conditions and show that they become equivalent in bipartite graphs. For general graphs, we provide tight structural characterizations, based on block and tight cut decompositions, under which even cycle symmetry and perfect matching equality force node-induced weights.

Finally, we resolve the motivating open question in the negative by constructing a matching-covered non-bipartite graph that satisfies the edge min-max property (every edge is contained in a minimum-weight perfect matching and a maximum-weight one) but violates perfect matching equality (all perfect matchings have the same weight). This counterexample shows that a single edge is not always sufficient to eliminate all minimum-weight or maximum-weight perfect matchings, thereby proving the tightness of the $2(\ell-1)$ bound for $\ell=2$. We also discuss extensions of this framework to $b$-factors and arborescences.

\medskip


\end{abstract}
%%%%%%%%%%%%%%%%%%%%%%%%%%%%%%%%
\newpage
\pagenumbering{roman}
\tableofcontents
\newpage
\pagenumbering{arabic}
\setcounter{page}{1}

%%%%%%%%%%%%%%%%%%%%%%%%%%%%%%%%
\section{Introduction}
\label{sec:intro}
%%%%%%%%%%%%%%%%%%%%%%%%%%%%%%%%

The \emph{exact weight perfect matching problem} asks, given an edge-weighted graph and a target value $y$, whether there exists a perfect matching whose weight is exactly $y$. Papadimitriou and Yannakakis~\cite{papadimitriou1982complexity} introduced this problem for graphs in which each edge has weight $0$ or $1$, calling it the \emph{exact matching} problem, and Mulmuley, Vazirani, and Vazirani~\cite{mulmuley1987matching} proposed a randomized polynomial-time algorithm for it. Since then, it has been a long-standing open problem whether a deterministic polynomial-time algorithm exists, although a very recent work by Du~\cite{du2026bipartite} claims a positive resolution in the bipartite case.

In contrast, minimum- and maximum-weight perfect matchings can be computed deterministically in polynomial time~\cite{edmonds1965maximum}. Between this fact and the situation of the exact weight perfect matching problem, a natural question arises as follows: what about a perfect matching of the second smallest or second largest weight? Specifically, let $x_1 < x_2 < \dots < x_q$ denote the distinct weights of perfect matchings, and we say that a perfect matching is \emph{$\ell$-th smallest} if its weight is $x_\ell$. Our objective is to compute an $\ell$-th smallest perfect matching for a given $\ell$. Note that this problem is different from the classical \emph{$K$-best enumeration} problem~\cite{lawler1972procedure}, which asks for $K$ distinct feasible solutions whose objective values are the $K$ smallest or largest, typically allowing ties to be broken in an arbitrary but consistent way. The following theorem of El Maalouly, Haslebacher, Taubner, and Wulf~\cite{maalouly2025finding} provides a structural characterization that leads to a simple XP algorithm for finding an $\ell$-th smallest perfect matching.

\begin{thm}[El Maalouly, Haslebacher, Taubner, Wulf]\label{thm:ESA2025}
    Let $G$ be an edge-weighted graph for which the possible weights of perfect matchings are $x_1 < x_2 < \dots < x_q$, and let $\ell\in[q]$.
    \begin{enumerate}[label=(\roman*)]\itemsep0em
        \item If $G$ is bipartite, there exists a subset of edges $F$ of size at most $\ell - 1$ such  that the minimum weight of a perfect matching containing $F$ is $x_\ell$.
        \item If $G$ is non-bipartite, there exists a subset of edges $F$ of size at most $2(\ell - 1)$ such that the minimum weight of a perfect matching containing $F$ is $x_\ell$.
    \end{enumerate}
\end{thm}

By this result, in order to compute an $\ell$-th smallest perfect matching, it suffices to compute a minimum-weight perfect matching in all the instances in which we specify an edge subset $F$ of size at most $\ell - 1$ in the bipartite case and at most $2(\ell - 1)$ in the non-bipartite case that must be included in the solution. 

Theorem~\ref{thm:ESA2025} provides a natural approach for the exact weight perfect matching problem. By symmetry, the result holds for $\ell$-th largest perfect matchings, allowing the interval $[x_\ell, x_{q-\ell+1}]$ to be computed via the same method. In particular, given a target value $y$, if a perfect matching of weight $y$ exists, the process eventually identifies an index $\ell$ such that $x_\ell = y$ or $x_{q-\ell+1} = y$. Otherwise, the procedure provides a certificate of non-existence by yielding an index $\ell$ such that either $x_\ell < y < x_{\ell+1}$ or $x_{q-\ell} < y < x_{q-\ell+1}$.

The bounds in Theorem~\ref{thm:ESA2025} are tight in the following sense: \cite{maalouly2025finding} also provided edge-weighted graphs for which one cannot guarantee a smaller forcing set than $\ell-1$ in the bipartite case and $2(\ell-1)$ in the non-bipartite case. In particular, even for $\ell=2$, there are non-bipartite graphs where any edge set excluding all minimum-weight perfect matchings must contain at least two edges. An analogous statement holds for maximum-weight perfect matchings, although the graphs witnessing tightness must be modified accordingly. However, it is unclear whether a single edge is sufficient in general to eliminate all minimum-weight perfect matchings or all maximum-weight perfect matchings by forcing that edge to be included. This raises the following question:

\begin{qu}\label{qu:origin}
    If the minimum and maximum weights of perfect matchings are different, does there exist an edge $f$ such that, among perfect matchings containing $f$, the minimum weight is strictly larger than the global minimum in $G$, or the maximum weight is strictly smaller than the global maximum in $G$?
\end{qu}
In particular, in the aforementioned application to the exact weight perfect matching problem, this weaker property -- the elimination of at least one extreme -- may already be sufficient. Motivated by this, we study the relations among several natural properties concerning the minimum and maximum weights of perfect matchings.

%%%%%%%%%%%%%%%%%%%%%%%%%%%%%%%%
\subsection{Our results}
%%%%%%%%%%%%%%%%%%%%%%%%%%%%%%%%

Let $G=(V,E)$ be a graph and let $w\colon E \to \mathbb{R}$ be an edge-weight function. Observe that there is a natural obstruction to an affirmative answer to Question~\ref{qu:origin}: if every edge $f\in E$ is contained in both a minimum-weight perfect matching and a maximum-weight perfect matching, then fixing any single edge cannot shift either extreme. This raises the question of understanding the relation between graphs in which all perfect matchings have the same weight and graphs in which every edge $f\in E$ is contained in both a minimum-weight perfect matching and a maximum-weight perfect matching. 

To this end, we consider the following properties of $w$.

\begin{symenum}
\itemsep0em
\itemsymbol{IND}\label{NIND} \textbf{Node-induced weights.}
The function $w$ is \emph{node-induced}, that is, there exists a function $f\colon V \to \mathbb{R}$ such that $w(uv)=f(u)+f(v)$ for all $uv\in E$.

\itemsymbol{EWS}\label{EWP} \textbf{Even walk symmetry.}
For every closed walk of even length, the total weight of the edges in odd positions equals that of the edges in even positions, counted with multiplicity.

\itemsymbol{ECS}\label{ECP} \textbf{Even cycle symmetry.}
For every even cycle $C$, the two perfect matchings of $C$ have equal total weight.

\itemsymbol{PME}\label{MEQ} \textbf{Perfect matching equality.}
All perfect matchings of $G$ have the same total weight.

\itemsymbol{EMX}\label{MMM} \textbf{Edge min-max.}
For every edge $e \in E$, there exist perfect matchings $A$ and $B$ containing $e$ such that $A$ has the minimum weight and $B$ has the maximum weight.
\end{symenum}

First, we establish a simple and basic hierarchy among these properties, showing that each condition is progressively weaker.

\begin{restatable}{thm}{thmtrivial}
\label{thm:trivial}
    Let $G=(V,E)$ be a matching-covered graph and let $w\colon E \to \mathbb{R}$ be an edge-weight function. Then \ref{NIND} implies \ref{EWP}, \ref{EWP} implies \ref{ECP}, \ref{ECP} implies \ref{MEQ}, and \ref{MEQ} implies \ref{MMM}.
\end{restatable}
Note that Question~\ref{qu:origin} is related to the contrapositive of the last implication in this hierarchy, in the sense that it asks whether the failure of \ref{MMM} can always be certified by the existence of an edge witnessing the failure of \ref{MEQ}.

Next, we characterize node-induced weights in terms of even walk symmetry.

\begin{restatable}{thm}{thmmainfirst}
\label{thm:main1}
    Let $G=(V,E)$ be a graph and let $w\colon E \to \mathbb{R}$ be an edge-weight function. Then \ref{NIND} and \ref{EWP} are equivalent.
\end{restatable}

It turns out that even cycle symmetry already forces a node-induced representation under an appropriate structural assumption.

\begin{restatable}{thm}{thmmainsecond}
\label{thm:main2}
Let $G=(V,E)$ be a graph.
\begin{enumerate}[label=(\roman*)]\itemsep0em
    \item If the block decomposition of $G$ contains at most one non-bipartite block, then every edge-weight function $w\colon E \to \mathbb{R}$ satisfying \ref{ECP} also satisfies \ref{NIND}. \label{it:2i}
    \item If the block decomposition of $G$ contains at least two non-bipartite blocks, then there exists an edge-weight function $w\colon E \to \mathbb{R}$ satisfying \ref{ECP} but not \ref{NIND}. \label{it:2ii}
\end{enumerate}
\end{restatable}

We identify a structural condition under which perfect matching equality forces node-induced weights.

\begin{restatable}{thm}{thmmainthird}
\label{thm:main3}
Let $G=(V,E)$ be a matching-covered graph.
\begin{enumerate}[label=(\roman*)]\itemsep0em
    \item If the tight cut decomposition of $G$ contains at most one brick, then every edge-weight function $w\colon E \to \mathbb{R}$ satisfying \ref{MEQ} also satisfies \ref{NIND}. \label{it:3i}
    \item If the tight cut decomposition of $G$ contains at least two bricks, then there exists an edge-weight function $w\colon E \to \mathbb{R}$ satisfying \ref{MEQ} but not \ref{NIND}.\label{it:3ii}
\end{enumerate}
\end{restatable}

Finally, we give a negative answer to Question~\ref{qu:origin}, showing that the obstruction in Theorem~\ref{thm:ESA2025} is tight already for $\ell=2$: even if the goal is to exclude all minimum-weight perfect matchings or all maximum-weight perfect matchings, one must specify two edges.

\begin{restatable}{thm}{thmmainfourth}
\label{thm:main4}
    There exists a matching-covered graph $G=(V,E)$ with an edge-weight function $w\colon E \to \mathbb{R}$ such that \ref{MMM} holds but \ref{MEQ} fails.
\end{restatable}

We then show that the behavior observed in Theorem~\ref{thm:ESA2025} is not specific to perfect matchings and also appears in other combinatorial optimization settings, such as $b$-factors and arborescences.

\begin{restatable}{thm}{thmbfacs}
\label{thm:b-factor}
    Let $G$ be an edge-weighted graph with a node-capacity function $b\colon V \to \mathbb{Z}_+$, for which the possible weights of $b$-factors are $x_{1} < x_{2} < \dots < x_{q}$ and let $\ell\in[q]$.
    \begin{enumerate}[label=(\roman*)]\itemsep0em
        \item If $G$ is bipartite, there exist edge sets $F$ and $D$ with $F\cap D = \emptyset$ and $|F| + |D| \leq \ell -1$ such that the minimum weight of a $b$-factor $S$ satisfying $F\subseteq S \subseteq E\setminus D$ is $x_\ell$.
        \item If $G$ is non-bipartite, there exist edge sets $F$ and $D$ with $F\cap D = \emptyset$ and $|F| + |D| \leq 2(\ell -1)$ such that the minimum weight of a $b$-factor $S$ satisfying $F\subseteq S \subseteq E\setminus D$ is $x_\ell$.
    \end{enumerate}
\end{restatable}

\begin{restatable}{thm}{thmarbs}
\label{thm:arbs}
    Let $D$ be a root-connected digraph with edge costs, and suppose that the possible costs of $r$-arborescences are
    $x_1 < x_2 < \dots < x_q$, and let $\ell \in [q]$. Then there exists a subset of edges $P$ with $|P| \le 2(\ell - 1)$ such that the minimum cost of an $r$-arborescence containing $P$ is $x_\ell$.
\end{restatable}

%%%%%%%%%%%%%%%%%%%%%%%%%%%%%%%%
\subsection{Paper organization}
\label{sec:organization}
%%%%%%%%%%%%%%%%%%%%%%%%%%%%%%%%

The rest of the paper is organized as follows. In Section~\ref{sec:preliminaries}, we introduce basic definitions and notation, and overview some results that we will use in our proofs. Section~\ref{sec:bipartite} proves Theorem~\ref{thm:trivial} and establishes the equivalence of the properties in question for matching-covered bipartite graphs. Section~\ref{sec:impliations} is devoted to the main results of the paper, Theorems~\ref{thm:main1}--\ref{thm:main4}. Finally, we extend Theorem~\ref{thm:ESA2025} to $b$-factors and to arborescences in Section~\ref{sec:bfacandarbs}.

%%%%%%%%%%%%%%%%%%%%%%%%%%%%%%%%
\section{Preliminaries}
\label{sec:preliminaries}
%%%%%%%%%%%%%%%%%%%%%%%%%%%%%%%%

We provide the basic definitions and notation here; all additional terminology is introduced when first needed.

\paragraph{Basic notation.} We denote the sets of \emph{reals} and \emph{integers} by $\bR$ and $\bZ$, respectively, and add $+$ as a subscript when restricted to nonnegative values. For a positive integer $k$, we use $[k]\coloneqq\{1,\dots,k\}$. Given a ground set $S$, the \emph{difference} of $X,Y\subseteq S$ is denoted by $X\setminus Y$. If $Y$ consists of a single element $y$, then $X\setminus \{y\}$ and $X\cup \{y\}$ are abbreviated as $X-y$ and $X+y$, respectively. Similarly, the single element set $\{y\}$ is often denoted by $y$. The \emph{characteristic vector} of $Y$ is denoted by $\chi_Y$, that is, $\chi_Y\in\{0,1\}^S$ with $\chi_Y(s)=1$ if $s\in Y$ and $0$ otherwise.  

For a vector space $V$, we denote its \emph{dimension} by $\dim V$. For a set $X\subseteq \bR^n$, the \emph{span} of $X$ is denoted by $\spa(X)$. Given a subspace $W\subseteq \bR^n$, its \emph{orthogonal complement} is denoted by $W^\perp$, that is, $W^\perp \coloneqq \{x\in \bR^n \colon x\cdot y = 0 \text{ for all } y\in W\}$, where $x \cdot y$ denote the standard inner product of $x$ and $y$. For a linear map $\varphi\colon\bR^m\to\bR^n$, we denote the \emph{kernel} and the \emph{image} of $\varphi$ by $\ker\varphi$ and $\im\varphi$, respectively. Recall that, by the dimension theorem, we have $\dim (\ker \varphi) + \dim (\im \varphi) = m$.

\paragraph{Graphs.} For a graph $G$, we denote its node set and edge set by $V(G)$ and $E(G)$, respectively. For two graphs $G_1$ and $G_2$, we say that $G_1$ is a \emph{subgraph} of $G_2$ if $V(G_1)\subseteq V(G_2)$ and $E(G_1)\subseteq E(G_2)$, and denote this by $G_1\subseteq G_2$. Let $G=(V,E)$ be a graph and let $X\subseteq V$. The subgraph of $G$ \emph{induced by $X$} is denoted by $G[X]$, while the \emph{graph obtained by deleting $X$} is denoted by $G-X$. 

A \emph{walk} is a sequence $W = (v_0, e_1, v_1, \dots, e_k, v_k)$ where each $e_i = v_{i-1}v_i \in E$. A walk may repeat nodes and edges. If $v_0 = v_k$, then $W$ is called a \emph{closed walk}. A \emph{path} is a walk in which all nodes, and hence all edges, are distinct, except possibly $v_0 = v_k$ in the case of a closed path which we usually call a \emph{cycle}. We will use the following result from~\cite{little2006parity}.

\begin{thm}[Little, Vince]\label{thm:even_odd}
    Let $C$ be an odd cycle in a $2$-connected graph $G$, and let $e \in E(G) \setminus E(C)$. Then $G$ contains both an even and an odd cycle through $e$ that also contains an edge of $C$.
\end{thm}

A \emph{matching} is a set of edges such that no two edges share a common endpoint. A matching is called \emph{perfect} if every node is incident to an edge in the matching. We say that a connected graph is \emph{matching-covered} if every edge is contained in a perfect matching. 

For a node subset $S \subseteq V$, the cut $\delta(S)$ is defined as the set of edges having exactly one endpoint in $S$. The sets $S$ and $V \setminus S$ are referred to as the \emph{shores} of the cut. A cut is termed \emph{odd} if the cardinality of both of its shores is odd. An odd cut $C = \delta(S)$ is considered \emph{tight} if it is nontrivial -- meaning $1 < |S| < |V| - 1$ -- and every perfect matching $M$ of $G$ contains exactly one edge from $C$ (that is, $|M \cap C| = 1$).

Given a nontrivial cut $C = \delta(S)$, the \emph{$C$-contractions} (or \emph{cut-contractions}) of $G$, conventionally denoted as $G/S$ and $G/(V \setminus S)$, are the graphs obtained by shrinking the shores $S$ and $V \setminus S$ into single nodes, respectively. A fundamental structural property is that whenever $G$ is a matching-covered graph and $C$ is a tight cut, both resulting $C$-contractions are guaranteed to remain matching-covered. 

\paragraph{Tight cut decompositions.} A matching-covered graph that does not admit any tight cuts is called a \emph{brick} if it is non-bipartite, and a \emph{brace} if it is bipartite. Whenever a matching-covered graph $G$ contains a tight cut $C$, it can be split into its two corresponding $C$-contractions. The recursive application of this process eventually yields a collection of bricks and braces, which cannot be decomposed further. This recursive procedure is called a \emph{tight cut decomposition}. A fundamental theorem by Lovász and Plummer~\cite{lovasz1986matching} guarantees that the final collection of bricks and braces is independent of the sequence in which the tight cuts are selected for contraction. Since the $C$-contraction operation preserves bipartiteness, a graph is bipartite if and only if its tight cut decomposition contains no bricks.

\paragraph{Ear decompositions.} Let $G$ be a graph and let $G'$ be a subgraph of $G$. A path $P$ is called an \emph{ear} of $G'$ if its endnodes belong to $G'$, while all internal nodes lie in $V(G)\setminus V(G')$. An ear is called \emph{odd} if it contains an odd number of edges. We say that an ear is \emph{closed} if its two endnodes coincide, and \emph{open} otherwise. 

An \emph{ear decomposition} of $G$ is a sequence of subgraphs
\[
G_{0} \subset G_{1} \subset \dots \subset G_{k} = G
\]
where $G_0$ is a cycle (which is not considered an ear), and for each $i \in \{0, \dots, k-1\}$, the graph $G_{i+1}$ is obtained from $G_i$ by adding an ear $P_{i+1}$ of $G_i$. 

A \emph{block} of $G$ is a maximal subgraph $B\subseteq G$ such that any two nodes in $B$ lie on a common cycle, or $B$ consists of a cut-edge together with its endpoints. The blocks of a graph partition its edge set, and the intersection of any two distinct blocks is either empty or a single cut-node; thus the blocks form a tree-like structure. The following well-known result is due to Whitney~\cite{whitney1931non} and Robbins~\cite{robbins1939theorem}.

\begin{thm}[Robbins, Whitney]\label{thm:whitney}
   A graph admits an ear decomposition if and only if it is 2-edge-connected, and an open ear decomposition if and only if it is 2-connected. Furthermore, for any 2-edge-connected graph $G$, there exists an ear decomposition starting from an arbitrary block $B_0$ and any cycle $C\subseteq B_0$ as the first ear, such that a closed ear is introduced if and only if the decomposition enters a new block.
\end{thm}

%%%%%%%%%%%%%%%%%%%%%%%%%%%%%%%%
\section{Basic implications and bipartite graphs}
\label{sec:bipartite}
%%%%%%%%%%%%%%%%%%%%%%%%%%%%%%%%

To better understand the properties studied in this paper, we first prove as a warm-up the easy implications given in Theorem~\ref{thm:trivial}, and then show that all the properties are equivalent for edge-weighted matching-covered \emph{bipartite} graphs.

%%%%%%%%%%%%%%%%%%%%%%%%%%%%%%%%
\subsection{Proof of Theorem~\ref{thm:trivial}}
%%%%%%%%%%%%%%%%%%%%%%%%%%%%%%%%

\thmtrivial*
\begin{proof}
    We prove the implications one by one.
    \medskip
    
    \noindent\textbf{\ref{NIND}$\implies$\ref{EWP}.}
    Let $f$ be the node-weight function that induces $w$. Let $P$ be an arbitrary even closed walk with edges $e_{1}, e_{2}, \dots, e_{2k}$ in order, and nodes $v_{0}, v_{1}, \dots, v_{2k}=v_{0}$. Then
    \[
    \sum_{i=1}^{k} w(e_{2i-1}) = \sum_{i=0}^{k-1} \bigl(f(v_{2i})+f(v_{2i+1})\bigr) = \sum_{i=1}^{k} \bigl(f(v_{2i-1})+f(v_{2i})\bigr)=\sum_{i=1}^{k} w(e_{2i}).
    \]
    \medskip
    
    \noindent\textbf{\ref{EWP}$\implies$\ref{ECP}.}
    This is immediate, since every even cycle is an even closed walk.
    \medskip
    
    \noindent\textbf{\ref{ECP}$\implies$\ref{MEQ}.}
    Let $M_1, M_2$ be perfect matchings. Then $M_1 \triangle M_2$ is a disjoint union of even cycles. On each such cycle, $M_1$ and $M_2$ restrict to the two perfect matchings of the cycle. By \ref{ECP}, these two matchings have equal weight on every cycle, hence $w(M_1)=w(M_2)$.
    \medskip
    
    \noindent\textbf{\ref{MEQ}$\implies$\ref{MMM}.}
    Since $G$ is matching-covered, every edge lies in a perfect matching. By \ref{MEQ}, all perfect matchings have the same weight, so every perfect matching containing a given edge is both minimum- and maximum-weight.
\end{proof}

%%%%%%%%%%%%%%%%%%%%%%%%%%%%%%%%
\subsection{Bipartite case}
%%%%%%%%%%%%%%%%%%%%%%%%%%%%%%%%

Consider the perfect matching polytope $P_{\mathrm{PM}}(G)$ of a matching-covered bipartite graph $G$, defined as the convex hull of the characteristic vectors of all perfect matchings of $G$.
\[
P_{\mathrm{PM}}(G) = \bigl\{ x \in \mathbb{R}^{E} \colon \sum_{e \in \delta(v)} x_e = 1\ \text{for all}\ v \in V,\ x \ge 0 \bigr\}.
\]
Therefore, viewing the maximum-weight perfect matching problem as a linear program, we can write the primal and dual problems as follows, where $A \in \{0, 1\}^{V \times E}$ is the incidence matrix of $G$:
\begin{center}
\begin{minipage}{0.35\textwidth}
\begin{tcolorbox}
\textbf{Primal}\vspace{-4mm}

\[
\begin{aligned}
\text{maximize} \quad & w\cdot x \\
\text{subject to} \quad & Ax = \mathbf{1} \\
& x \ge \mathbf{0}
\end{aligned}
\]
\end{tcolorbox}
\end{minipage}
\hspace{0.1\textwidth}
%\hfill
\begin{minipage}{0.35\textwidth}
\begin{tcolorbox}[valign=top]
\textbf{Dual}\vspace{-4mm}

\[
\begin{aligned}
\text{minimize} \quad & \mathbf{1}\cdot y \\
\text{subject to} \quad & yA \ge w
\end{aligned}
\]
\end{tcolorbox}
\end{minipage}
\end{center}
The following observation then follows easily.

\begin{thm}\label{thm:bipartite}
Let $G$ be a matching-covered bipartite graph. Then \ref{NIND}, \ref{EWP}, \ref{ECP}, \ref{MEQ}, and \ref{MMM} are equivalent.
\end{thm}

\begin{proof}
    By \cref{thm:trivial}, it suffices to show that \ref{MMM} implies \ref{NIND}. Let $x$ and $y$ be a pair of optimal primal and dual solutions, respectively. By complementary slackness, for every edge $uv \in E$ that is contained in a maximum-weight perfect matching, we have
    \[
       y_u + y_v = w_{uv}.
    \]
    Since $G$ satisfies \ref{MMM} and the graph is matching-covered, every edge belongs to a maximum-weight perfect matching, hence this equality holds for every edge $uv \in E$. Thus the dual optimal solution $y$ induces the edge-weight function $w$, that is, $w(uv) = y_u + y_v$ for all $uv \in E$, which shows that $w$ is node-induced. This completes the proof.
\end{proof}

%%%%%%%%%%%%%%%%%%%%%%%%%%%%%%%%
\section{Implications between the main properties}
\label{sec:impliations}
%%%%%%%%%%%%%%%%%%%%%%%%%%%%%%%%

This section is devoted to the proofs of our main results, Theorems~\ref{thm:main1}-\ref{thm:main4}. 

%%%%%%%%%%%%%%%%%%%%%%%%%%%%%%%%
\subsection{General graphs}
\label{sec:general}
%%%%%%%%%%%%%%%%%%%%%%%%%%%%%%%%

Our first result shows that \ref{EWP} and \ref{NIND} are in fact equivalent.

\thmmainfirst*
\begin{proof}
    By Theorem~\ref{thm:trivial}, it suffices to prove that \ref{EWP} implies \ref{NIND}. Let $G=(V,E)$ be a graph and let $w\colon E \to \mathbb{R}$ satisfy \ref{EWP}. We may assume that $G$ is connected, since otherwise the argument can be applied independently on each connected component. 
        
    If $G$ is non-bipartite, then it contains an odd cycle. Moreover, for every $v \in V$ there exists an odd closed walk $W$ that visits $v$, obtained as follows: take an odd cycle $C$, choose a path from $v$ to $C$, and traverse this path to the cycle and back before traversing $C$. The key idea is to construct a potential via alternating sums along odd walks and to show that it is well defined using \ref{EWP}. More precisely, for every node $v\in V$, fix an odd closed walk $W$ starting and ending at $v$, and write its edges in order as $e_{1}, \dots, e_{2k+1}$ where the walk starts at $v$, ends at $v$, and $e_{2k+1}$ is the last edge returning to $v$. Define $f(v) \coloneqq \frac{1}{2}\sum_{i = 1}^{k} (-1)^{i+1}w(e_{i})$. 

    We first show that $f(v)$ is well defined, i.e., independent of the choice of the walk. Let $W$ and $W'$ be two odd closed walks visiting $v$, with edge sequences $e_1,\dots,e_{2k+1}$ and $e'_1,\dots,e'_{2\ell+1}$, respectively. Then the concatenation $W \circ W'$ is an even closed walk. By \ref{EWP}, the alternating sum of edge weights over this walk is zero, hence
    \[
        0=\sum_{i=1}^{2k+1} (-1)^{i+1}w(e_{i})-\sum_{i=1}^{2\ell+1} (-1)^{i+1}w(e'_{i}).
    \]
    This shows that the value of $f(v)$ does not depend on the chosen odd closed walk.

    We claim that $f$ induces $w$. Let $uv$ be an arbitrary edge of the graph. Let $W$ be an odd closed walk starting and ending at $u$ with edge sequence $e_{1},\dots,e_{2k+1}$. We construct another odd walk $W'$ starting and ending at $v$ by adding two copies of the edge $uv$ to $W$, that is, let $e_{0}=e_{2k+2}=uv$. Then
    \[
        f(u) + f(v) 
        = \frac{1}{2}\sum_{i = 1}^{2k+1} (-1)^{i+1}w(e_{i})+\frac{1}{2}\sum_{i = 0}^{2k+2} (-1)^{i}w(e_{i})
        =\frac{1}{2}(w(e_0)+w(e_{2k+2}))
        =w(uv),
    \]
    which completes the proof if $G$ is non-bipartite.

    If $G$ is bipartite, we proceed in a similar way. Fix an arbitrary node $v \in V$ and set $f(v)$ to be an arbitrary real value. For every node $u\in V-v$, fix a $uv$-path $P$ and let $e_{1}, \dots, e_{k}$ be its edges in order where $e_1$ is incident to $u$. Define $f(u) \coloneqq \frac{1}{2}\sum_{i = 1}^{k} (-1)^{i+1}w(e_{i})+(-1)^k f(v)$. 

    We first show that $f(u)$ is well defined, i.e., independent of the choice of the path. Let $P$ and $P'$ be two $uv$-paths with edge sequences $e_{1}, \dots, e_{k}$ and $e'_{1}, \dots, e'_{\ell}$, respectively. Since $G$ is bipartite, all $uv$-paths have the same parity, hence $k \equiv \ell \pmod 2$. Therefore, concatenating $P$ with the reverse of $P'$ results in an even closed walk. By \ref{EWP}, the alternating sum of edge weights over this walk is zero, hence
    \begin{align*}
        0
        &=
        \sum_{i=1}^{k} (-1)^{i+1} w(e_i)
        +\sum_{i=1}^{\ell} (-1)^{k+i+1} w(e'_{\ell+1-i})\\
        &=
        \sum_{i=1}^{k} (-1)^{i+1} w(e_i)
        +\sum_{i=1}^{\ell} (-1)^{k+\ell-i} w(e'_{i})\\
        &=
        \sum_{i=1}^{k} (-1)^{i+1} w(e_i)
        -\sum_{i=1}^{\ell} (-1)^{i+1} w(e'_{i}).
    \end{align*}
    Since $(-1)^k f(v)=(-1)^\ell f(v)$, this shows that the value of $f(u)$ does not depend on the chosen $uv$-path.
    
    We now show that $f$ induces $w$. Let $st \in E$. Let $P_{tv}$ be a $tv$-path, and let $P_{sv}$ be the $sv$-path obtained by prepending the edge $st$ to $P_{tv}$. Writing $P_{tv}$ as $e_{1}, \dots, e_{k}$, we obtain
    \[
        f(s) = w(st) + \sum_{i=1}^{k} (-1)^{i} w(e_i) + (-1)^{k+1} f(v),
    \quad
    f(t) = \sum_{i=1}^{k} (-1)^{i+1} w(e_i) + (-1)^k f(v),
    \]
    and therefore $f(s)+f(t)=w(st)$ as required.
\end{proof}

%%%%%%%%%%%%%%%%%%%%%%%%%%%%%%%%
\subsection{Graphs with at most one non-bipartite block}
\label{sec:noodd}
%%%%%%%%%%%%%%%%%%%%%%%%%%%%%%%%

Let $G=(V,E)$ be a graph and $w\colon E\to\bR$ be an edge-weight function. Then, we may view $w$ as a vector in $\mathbb{R}^{E}$, where each coordinate corresponds to an edge. This raises the natural question of determining the dimension of the space of node-induced edge-weights. We denote the space of edge-weight functions satisfying \ref{NIND} by $\mathcal{N}_G$. 

\begin{lem}\label{lem:space}
    Let $G = (V, E)$ be a connected graph. Then $\dim(\mathcal{N}_G) = |V| - 1$ if $G$ is bipartite and $\dim(\mathcal{N}_G) = |V|$ otherwise.
\end{lem}

\begin{proof}
    Let $f$ be a node-weight function, which we view as a vector in $\mathbb{R}^{V}$. Such an $f$ induces an edge-weight function $w$ defined by $w(uv)=f(u)+f(v)$ for every $uv\in E$. This defines a linear map $\varphi\colon\mathbb{R}^{V}\to\mathbb{R}^{E}$, mapping $f$ to $w$. By definition, $\mathcal{N}_G=\im \varphi$. Moreover, $\dim( \mathbb{R}^{V})=\dim(\ker \varphi)+\dim(\im \varphi)$, hence $\dim (\mathcal{N}_G)=|V|-\dim(\ker \varphi)$. Thus it suffices to determine $\ker \varphi$, i.e., the space of node-weights $f$ such that $f(u)+f(v)=0$ for every $uv\in E$.
    
    First assume that $G$ is non-bipartite. Then $G$ contains an odd cycle $C$. Let $f\in\ker \varphi$. For every edge $uv\in E(C)$ we have $f(u)=-f(v)$, hence the values of $f$ alternate along $C$. Since $C$ has odd length, this forces $f$ to take the same value and its negative at the starting node, hence $f\equiv 0$ on $V(C)$. Now let $x\in V$ be arbitrary. Since $G$ is connected, there exists a path from $x$ to $C$. Along each edge $uv$ of this path we have $f(u)=-f(v)$, so values propagate along the path. Since all nodes of $C$ have value $0$, it follows that $f(x)=0$. Hence $f\equiv 0$, and $\dim(\ker \varphi)=0$.
    
    Now assume that $G$ is bipartite with partition classes $A$ and $B$. Let $f\in\ker \varphi$. Then $f(u)=-f(v)$ for every edge $uv$ with $u\in A$ and $v\in B$. Fix $v_0\in A$ and set $x\coloneqq f(v_0)$. Let $u\in A$, and let $P$ be a path from $v_0$ to $u$. Since $G$ is bipartite, $P$ alternates between $A$ and $B$, and repeated use of $f(u)=-f(v)$ implies $f(u)=x$. Similarly, for every $v\in B$ we obtain $f(v)=-x$. Conversely, every function of this form satisfies $f(u)+f(v)=0$ on every edge, hence lies in $\ker \varphi$. Therefore $\dim(\ker \varphi)=1$.
    
    Combining both cases with $\dim (\mathcal{N}_G)=|V|-\dim(\ker \varphi)$ yields the statement of the lemma.    
\end{proof}

The following is a direct corollary.

\begin{cor}\label{cor:injective}
    Let $G = (V, E)$ be a connected non-bipartite graph. Consider the linear map $\varphi\colon\bR^V\to\bR^E$ defined by $(\varphi f)(uv)\coloneqq f(u)+f(v)$ for all $uv\in E$. Then $\varphi$ is injective.
\end{cor}

\begin{proof}
    The statement follows from the fact that $\dim (\ker \varphi) = 0$ when $G$ is non-bipartite. 
\end{proof}

Another simple corollary is as follows.

\begin{cor}\label{cor:jgnofg}
Let $G$ be a loopless graph. Then $\dim(\cN_G) = |V| - p(G)$, where $p(G)$ denotes the number of bipartite components of $G$.
\end{cor}

\begin{proof}
    The statement follows by applying Lemma~\ref{lem:space} to each connected component of $G$.
\end{proof}

We now turn to the second property, \ref{ECP}. Let $C$ be an even cycle of $G$, and let $M_1$ and $M_2$ denote the two perfect matchings of $C$. Then \ref{ECP} for $C$ is equivalent to
\[
w \cdot (\chi_{M_1} - \chi_{M_2}) = 0.
\]
For an even cycle $C$, we define its \emph{alternating characteristic vector} as $\nu_C \coloneqq \chi_{M_1} - \chi_{M_2}$, where $M_1$ and $M_2$ are the two perfect matchings of $C$. Note that $\nu_C$ is defined up to sign, but this ambiguity will not affect our arguments. Thus, $w$ satisfies \ref{ECP} if and only if
\[
    w \cdot \nu_C = 0 
\]
for every even cycle $C$. This naturally leads to the following definition. For a graph $G$, we define its \emph{alternating cycle space} as
\[
    \cC_G \coloneqq \spa\{\nu_C \colon C \text{ is an even cycle of } G\}.
\]

Our proof relies on the following two technical lemmas.

\begin{lem}\label{lem:C_subset_of_N_perp}
    For any graph $G$, we have $\cC_G \subseteq \cN_G^{\perp}$.
\end{lem}
\begin{proof}
    Let $w \in \cN_G$. Then $w$ satisfies \ref{NIND}, and hence, by \cref{thm:trivial}, it also satisfies \ref{ECP}. Therefore $w \cdot \nu_C = 0$ for every even cycle $C$, which implies that $w \in \cC_G^{\perp}$. Thus $\cN_G \subseteq \cC_G^{\perp}$, and taking orthogonal complements yields $\cC_G \subseteq \cN_G^{\perp}$.
\end{proof}

We now turn to the structural consequence of having a parity-consistent ear decomposition. In this setting, we show that the orthogonal complement of the alternating cycle space coincides exactly with the space of node-induced weight functions.

\begin{thm}\label{thm:ear_decomposition_main}
    Let $G$ be a graph whose block decomposition contains at most one non-bipartite block. Then $\cC_G = \cN_G^{\perp}$.
\end{thm}
\begin{proof}
    By \cref{lem:C_subset_of_N_perp}, it is enough to show that $\dim(\mathcal{C}_G) + \dim(\mathcal{N}_G) = |E(G)|$. We may assume that the graph is 2-edge-connected, since if the equality holds for every 2-edge-connected component, then adding each cut edge increases the number of edges by one and decreases the number of bipartite components by one, while the dimension of $\mathcal{C}_G$ remains unchanged. Hence, the required equality follows by Corollary~\ref{cor:jgnofg}.
    
    Therefore, assume that $G$ is 2-edge-connected. We prove this by induction on the number of ears in a carefully chosen ear decomposition given by Theorem~\ref{thm:whitney}: the decomposition is selected so that, if $G$ contains a non-bipartite block, the construction starts within such a block and begins with an odd cycle in that block.

    The base case is $G$ being a cycle. By \cref{lem:space}, $\dim(\cN_G) = |V(G)|-1$ or $|V(G)|$, depending on whether this cycle has even or odd length, respectively. If $G$ is an even cycle, then $\cC_G$ is spanned by its alternating characteristic vector, so $\dim(\cC_G)=1$. Otherwise, there are no even cycles, hence $\dim(\cC_G)=0$. In both cases, the equality $\dim(\cC_G) + \dim(\cN_G) = |E(G)|$ holds, since $|E(G)|=|V(G)|$.
    
    Now assume that $G_{i+1}$ is obtained from $G_i$ by adding one ear, and that the statement holds for $G_i$, i.e. $\cC_{G_i}^{\perp} = \cN_{G_i}$. We first observe that any even cycle of $G_i$ is also an even cycle of $G_{i+1}$, hence $\cC_{G_i} \subseteq \cC_{G_{i+1}}$. Let the added ear introduce $\ell$ new nodes, so that
    $|V(G_{i+1})| = |V(G_i)| + \ell$ and $|E(G_{i+1})| = |E(G_i)| + \ell + 1$. We distinguish two cases depending on bipartiteness.
    \medskip
    
    \noindent\textbf{Case 1:} $G_i$ is bipartite.
    
    If $G_i$ is bipartite, then $G$ is bipartite, since otherwise the decomposition would start with an odd cycle in a non-bipartite block. Hence $G_{i+1}$ is bipartite. By Lemma~\ref{lem:space}, we have $\dim(\cN_{G_{i+1}}) = |V(G_{i+1})| - 1 = |V(G_i)| + \ell - 1$. Since $G_i$ is also bipartite, we similarly have $\dim(\cN_{G_i}) = |V(G_i)| - 1$, which by the induction hypothesis means $\dim(\cC_{G_i}) = |E(G_i)| - |V(G_i)| + 1$. Since $\cC_{G_i} \subseteq \cC_{G_{i+1}} \subseteq \cN_{G_{i+1}}^{\perp}$, we can bound the dimension:
    \begin{align*}
    \dim(\cC_{G_{i+1}}) 
    &\le |E(G_{i+1})| - \dim(\cN_{G_{i+1}})\\
    &= (|E(G_i)| + \ell + 1) - (|V(G_i)| + \ell - 1)\\
    &= |E(G_i)| - |V(G_i)| + 2\\
    &= \dim(\cC_{G_i}) + 1.
    \end{align*}
    We show that $\cC_{G_{i+1}} \neq \cC_{G_i}$. For this, it is enough to find an even cycle in $G_{i+1}$ that uses the new ear. If the ear is closed, then it is itself an even cycle. Otherwise, let its endpoints be $x,y$. Since $G_{i+1}$ remains bipartite, every $xy$-path in $G_i$ has the same parity, hence combining the ear with any $xy$-path yields an even cycle containing the ear.
    \medskip

    \noindent\textbf{Case 2:} $G_i$ is non-bipartite.

    By Lemma~\ref{lem:space}, we have $\dim(\cN_{G_{i+1}}) = |V(G_{i+1})| = |V(G_i)| + \ell$. Since $G_i$ is also non-bipartite, we similarly have $\dim(\cN_{G_i}) = |V(G_i)|$, which by the induction hypothesis means $\dim(\cC_{G_i}) = |E(G_i)| - |V(G_i)|$. Since $\cC_{G_i} \subseteq \cC_{G_{i+1}} \subseteq \cN_{G_{i+1}}^{\perp}$, we can bound the dimension:
    \begin{align*}
    \dim(\cC_{G_{i+1}}) 
    &\le |E(G_{i+1})| - \dim(\cN_{G_{i+1}})\\
    &= (|E(G_i)| + \ell + 1) - (|V(G_i)| + \ell)\\
    &= |E(G_i)| - |V(G_i)| + 1\\
    &= \dim(\cC_{G_i}) + 1.
    \end{align*}
    We show that $\cC_{G_{i+1}} \neq \cC_{G_i}$. For this, it is enough to find an even cycle in $G_{i+1}$ that uses the new ear. If the ear is closed, then it is itself an even cycle as, by the choice of the decomposition, it is part of a bipartite block of $G$. Otherwise, let its endpoints be $x,y$. By Theorem~\ref{thm:even_odd}, there exists an even cycle in $G_{i+1}$ that uses the new ear. Thus $\dim(\cC_{G_{i+1}}) = \dim(\cC_{G_i}) + 1$. Therefore, $\dim(\cC_{G_{i+1}}) + \dim(\cN_{G_{i+1}}) = (|E(G_i)| - |V(G_i)| + 2) + (|V(G_i)| + \ell - 1) = |E(G_{i+1})|$.
    \medskip
    
    In all cases, the inductive step preserves $\dim(\cC_{G_{i+1}}) + \dim(\cN_{G_{i+1}}) = |E(G_{i+1})|$, which completes the proof.
\end{proof}

Now we are ready to prove Theorem~\ref{thm:main2}.

\thmmainsecond*
\begin{proof}
    For \ref{it:2i}, note that $\mathcal{C}_G = \mathcal{N}_G^{\perp}$ by Theorem~\ref{thm:ear_decomposition_main}. Hence every edge-weight function $w \colon E \to \mathbb{R}$ satisfying \ref{ECP} lies in $\mathcal{C}_G^\bot = \mathcal{N}_G$, and therefore satisfies \ref{NIND}.

    For \ref{it:2ii}, let $G$ be a graph whose block decomposition contains at least two non-bipartite blocks. Choose two non-bipartite blocks $B_{1}$ and $B_{2}$ of $G$ such that in the block tree of $G$ every node on the unique path between $B_{1}$ and $B_{2}$ corresponds to a bipartite block of $G$. Consider the following edge-weight function:
    \[
        w(e) =
        \begin{cases}
            1 & \text{if } e \in E(B_{1}),\\
            2 & \text{if } e \in E(B_{2}),\\
            0 & \text{otherwise}.
        \end{cases}
    \]
    First, we show that $w$ satisfies \ref{ECP}. It is clear that no cycle contains edges from more than one block. Therefore, if an edge-weight function is constant within each block it satisfies \ref{ECP}, which holds for $w$. 
    Second, we show that no node-weight function induces $w$. Suppose for contradiction that $f\colon V \rightarrow \mathbb{R}$ induces $w$. It is clear that ${f|_{V(B_{1})}} \equiv \frac{1}{2}$ induces $w|_{E(B_{1})}$. Moreover, since $B_{1}$ is non-bipartite, by \cref{cor:injective} this is also unique. Similarly, ${f|_{V(B_{2})}} \equiv 1$ is the unique node-weight function that induces $w|_{E(B_{2})}$. Consider the subgraph $H$ of $G$ which is the union of the blocks of $G$ corresponding to the nodes of the unique path between $B_{1}$ and $B_{2}$ in the block tree of $G$. Since $H$ is the union of some bipartite blocks, $H$ is also bipartite. Therefore, as we have seen in the proof of \cref{lem:space}, whenever a node-weight function induces ${w|_H} \equiv 0$, its absolute value must be constant. However, both $B_{1}$ and $B_{2}$ by definition share a node with $H$, denote this nodes by $b_{1}$ and $b_{2}$. We showed that $f(b_{1}) = \frac{1}{2}$ and $f(b_{2}) = 1$. Since $b_{1}$ and $b_{2}$ are nodes of $H$, this contradicts that the absolute value of $f|_H$ is constant.
\end{proof}

%%%%%%%%%%%%%%%%%%%%%%%%%%%%%%%%
\subsection{Matching-covered graphs}
\label{sec:mcovered}
%%%%%%%%%%%%%%%%%%%%%%%%%%%%%%%%

The goal of this section is to clarify the connection between \ref{NIND} and \ref{MEQ}. Since the latter property is unaffected by edges not contained in any perfect matching, we restrict our attention to matching-covered graphs. 

Let $G=(V,E)$ be a matching-covered graph. We denote the space of edge-weight functions satisfying \ref{MEQ} by $\mathcal{M}_G$, while the space generated by the incidence vectors of perfect matchings is denoted by $\cP_G$. In \cite{lovasz1987matching}, Lovász proved the following.

\begin{thm}[Lovász]\label{thm:lovasz}
    For a matching-covered graph $G$, the dimension of $\cP_G$ is $|E(G)|-|V(G)|+2-b(G)$, where $b(G)$ is the number of bricks in the tight cut decomposition of $G$. 
\end{thm}

Using this theorem, we are ready to prove our result.

\thmmainthird*

\begin{proof}
    By \cref{thm:trivial}, we have $\cN_G \subseteq \cM_G$. Therefore, $\cN_G = \cM_G$ holds if and only if $\dim(\cN_G) = \dim(\cM_G)$. A weight function $w$ belongs to $\cM_G$ if and only if $w \cdot (\chi_{M_1} - \chi_{M_2}) = 0$ for all perfect matchings $M_1, M_2$ of $G$. In other words, $\cM_G$ is the orthogonal complement of the difference space of $\cP_G$. Since this difference space has dimension $\dim(\cP_G) - 1$, we have $\dim(\cM_G) = |E| - (\dim(\cP_G) - 1)$. Applying Theorem~\ref{thm:lovasz}, we obtain:
    \[
        \dim(\cM_G) = |E| - (|E| - |V| + 2 - b(G) - 1) = |V| - 1 + b(G).
    \]
    
    If $G$ is bipartite, then $b(G) = 0$. In this case, \cref{lem:space} gives $\dim(\cN_G) = |V| - 1$, which equals $\dim(\cM_G)$, implying $\cN_G = \cM_G$.

    If $G$ is non-bipartite, then $b(G) \geq 1$ and \cref{lem:space} gives $\dim(\cN_G) = |V|$. In this case, the dimensions match if and only if $|V| = |V| - 1 + b(G)$, which simplifies to $b(G) = 1$. Thus, $\cN_G = \cM_G$ if and only if the tight cut decomposition of $G$ contains exactly one brick, which proves both \ref{it:3i} and~\ref{it:3ii}.
\end{proof}

\begin{rem}
    Let $G$ be a graph whose tight cut decomposition contains at least two bricks. Then we can explicitly construct an edge-weight function that satisfies~\ref{MEQ} but violates~\ref{NIND}.
    
    Indeed, let $C$ be a tight cut with the property that none of the corresponding $C$-contractions is bipartite; such a tight cut exists by $b(G)\geq 2$. Define $w \in \mathbb{R}^{E(G)}$ by
    \begin{equation*}
    w(e) =
    \begin{cases}
        1 & \text{if } e \in \delta(C), \\
        0 & \text{otherwise.}
    \end{cases}
    \end{equation*}
    Since $C$ is tight, every perfect matching of $G$ contains exactly one edge of $\delta(C)$, and hence every perfect matching has total weight $1$. Thus $w$ satisfies~\ref{MEQ}.
    
    On the other hand, $w$ is not node-induced. Indeed, both sides of the cut $C$ are non-bipartite, so no node weight function can induce $w$. Any such node weighting would have to assign value $0$ to all nodes, which would imply that all edges have weight $0$, contradicting the fact that edges in $\delta(C)$ have weight $1$.
\end{rem}

%%%%%%%%%%%%%%%%
\subsection{Resolving Question~\ref{qu:origin}}
\label{sec:negative}
%%%%%%%%%%%%%%%%

Finally, we answer Question~\ref{qu:origin} in the negative.

\thmmainfourth*

\begin{proof}
    We provide a constructive counterexample. Let $V$ be partitioned into four sets $V_1, V_2, V_3, V_4$, where $V_1$ and $V_3$ each induces a triangle, and $V_2$ and $V_4$ are independent sets of size two. We add all edges between $V_i$ and $V_{i+1}$ for $i \in \{1, 2, 3, 4\}$, where indices are considered cyclically. It is easy to check that the resulting graph $G$ is matching-covered and non-bipartite.

    Let $S = V_1 \cup V_2$. Note that $|S| = 3 + 2 = 5$, so the cut $C = \delta(S)$ is an odd cut. We define the edge-weight function $w \in \mathbb{R}^E$ as the characteristic vector of $C$:
    \begin{equation*}
    w(e) = 
    \begin{cases}
        1 & \text{if $e \in C$}, \\
        0 & \text{otherwise}. 
    \end{cases}
    \end{equation*}
    It is straightforward to verify that every edge is contained in a perfect matching of weight $1$ and also in one of weight $3$, and moreover $1$ and $3$ are respectively the minimum and maximum weights of a perfect matching. Therefore, $w$ satisfies \ref{MMM} but \ref{MEQ} fails as stated.  
\end{proof}

%%%%%%%%%%%%%%%%%%%%%%%%%%%%%%%%
\section{$b$-factors and arborescences}
\label{sec:bfacandarbs}
%%%%%%%%%%%%%%%%%%%%%%%%%%%%%%%%

The proof of Theorem~\ref{thm:ESA2025} in~\cite{maalouly2025finding} (with symmetry of minimum and maximum) is based on the observation that, when the maximum-weight matching problem is formulated as a linear program and its dual is considered, the complementary slackness conditions imply that a perfect matching is of maximum weight if and only if it consists of so-called tight edges, and it enters each member of a certain laminar family (the support of an optimal dual solution) with exactly one edge. This means that, if we wish to fix edges in such a way that the weight of a maximum-weight perfect matching decreases, it suffices to fix either a non-tight edge, or two tight edges that both enter the same member of the laminar family. In this section, we extend this idea in two directions: minimum-weight $b$-factors and minimum-cost arborescences.

%%%%%%%%%%%%%%%%%%%%%%%%%%%%%%%%
\subsection{$b$-factors}
\label{sec:bfac}
%%%%%%%%%%%%%%%%%%%%%%%%%%%%%%%%

The extension to $b$-factors is based on the close structural relationship between $b$-factors and perfect matchings. Take a graph $G=(V,E)$ and a node-capacity function $b \colon V \to \mathbb{Z}_+$. We say that an edge set $S \subseteq E$ is a \textit{$b$-factor} if every node $v \in V$ is incident exactly $b(v)$ edges of $S$. It is clear that a perfect matching is a $\mathbf{1}_V$-factor, where $\mathbf{1}_V$ denotes the all-one vector on $V$.

We give the following extension of Theorem~\ref{thm:ESA2025}.

\thmbfacs*

\begin{proof}
    We first recall the connection between $b$-factors and perfect matchings, see e.g.~\cite{schrijver2003combinatorial}.
    
    Given an edge-weighted graph $G = (V, E)$ and $b\colon V \rightarrow \mathbb{Z}_+$, we will construct a graph $G' = (V', E')$ such that a $b$-factor of $G$ corresponds to a perfect matching of $G'$ and vice versa. For each node $v\in V$, the graph $G'$ contains $b(v)$ many copies $v_{1}, \dots, v_{b(v)}$ of $v$. For each edge $e = uv \in E$, we introduce two new nodes $p_{e, u}$ and $p_{e, v}$ connected by an \emph{inner edge} $p_{e, u}p_{e, v}$. We also introduce \emph{outer edges} $u_{i}p_{e, u}$ for all $i \in [b(u)]$ and $v_{j}p_{e, v}$ for all $j \in [b(v)]$. For an edge $e = uv \in E$, the subgraph spanned by $u_{1}, \dots , u_{b(u)}, p_{e, u}, p_{e, v}, v_{1}, \dots , v_{b(v)}$ is called the \emph{corresponding gadget} of $e$. 
    
    Given a $b$-factor $S$ of $G$, we can construct a corresponding perfect matching $P$ of $G'$ as follows. If $e \notin S$, we pick its inner edge $p_{e, u}p_{e, v}$ in $G'$. If $e = uv \in S$, we pick one outer edge incident to $p_{e, u}$ and one outer edge incident to $p_{e, v}$. Since $S$ is a $b$-factor, exactly $b(v)$ edges of $S$ are incident to each node $v$. Because the outer edges form a complete bipartite graph between $\{v_{j} \colon j \in [b(v)]\}$ and $\{p_{e, v} \colon e \text{ is incident to } v\}$, we can easily route these choices such that every copy $v_j$ is covered exactly once. This logic holds in reverse, meaning every perfect matching of $G'$ can be represented as a $b$-factor of $G$. An important property of this construction is that $G'$ is bipartite if and only if $G$ is bipartite. 
    
    Given a weight function $w\colon E \rightarrow \mathbb{R}$, we define a weight function $w'$ on the edges of $G'$. For all outer edges, let $w'(u_{i}p_{e, u}) \coloneqq w(e)/2$ and $w'(v_{j}p_{e, v}) \coloneqq w(e)/2$. For all inner edges, let $w'(p_{e, u}p_{e, v}) \coloneqq 0$. When an edge $e=uv \in S$, the corresponding perfect matching $P$ selects exactly two outer edges associated with $e$, accumulating a total weight of $w(e)/2 + w(e)/2 = w(e)$. If $e \notin S$, $P$ selects the inner edge, contributing a weight of $0$. Thus, a $b$-factor of $G$ has the exact same weight as its corresponding perfect matching in $G'$.
    
    Now we are ready to prove the statement of the theorem. We consider the case where $G$ is bipartite; the non-bipartite case works similarly. Construct $G'$ and $w'$ as previously described. Since a  $b$-factor has the same weight as its corresponding perfect matching, the possible weights of the perfect matchings of $G'$ are also $x_{1} < x_{2} < \dots < x_{q}$. We can therefore apply \cref{thm:ESA2025}. This implies there exists an edge set $A \subseteq E'$ of size at most $\ell-1$ such that the minimum weight of a perfect matching containing $A$ is $x_{\ell}$. 
    
    From $A$, we determine the sets $F$ and $D$ for our original graph $G$. We assign an edge $e = uv \in E$ to $D$ if $p_{e, u}p_{e, v} \in A$. We assign $e = uv \in E$ to $F$ if $A$ contains at least one of its outer edges, that is, if $A \cap (\{u_{i}p_{e, u} \mid i \in [b(u)]\} \cup \{v_{j}p_{e, v} \mid j \in [b(v)]\}) \neq \emptyset$. From the correspondence described earlier, the perfect matchings of $G'$ containing $A$ correspond exactly to the $b$-factors $S$ of $G$ satisfying $F \subseteq S \subseteq E \setminus D$. Furthermore, because $A$ is a subset of a valid matching in $G'$, it cannot contain both an inner edge and an outer edge of the same gadget. Thus, each edge in $A$ contributes to at most one element in $D \cup F$, which yields $|D|+|F| \leq |A| \leq \ell-1$, completing the proof.
\end{proof}

%%%%%%%%%%%%%%%%%%%%%%%%%%%%%%%%
\subsection{Arborescences}
\label{sec:arbs}
%%%%%%%%%%%%%%%%%%%%%%%%%%%%%%%%

The extension to arborescences is based on an LP-based characterization of minimum-cost arborescences.
An \textit{arborescence} is a directed tree in which every node has in-degree $1$, except for a unique node, called the \textit{root}, which has in-degree $0$. An $r$-\textit{arborescence} of a digraph is a spanning subgraph that is an arborescence rooted at a specified node $r$.

Let $D=(V,A)$ be a digraph with a fixed root $r \in V$. Let $\mathcal{A}_{D}$ denote the set of all $r$-arborescences of $D$. An edge $uv \in A$ is said to \textit{enter} a set $Z \subseteq V$ if $u \notin Z$ and $v \in Z$. For $Z \subseteq V$ and $F\subseteq A$, let $d^\mathrm{in}_F(Z)$ denote the number of edges entering $Z$ in $F$. We say that $D$ is \textit{root-connected} if $d^\mathrm{in}_A(Z) \ge 1$ holds for every nonempty set $Z \subseteq V - {r}$. This is equivalent to requiring that every node is reachable from $r$, and also to the existence of an $r$-arborescence in $D$.

Suppose that $D$ is root-connected.
Fix an edge $e = uv \in A$, and let $D'$ be the graph obtained from $D$ by removing all edges except $e$ that enter $v$ and contracting $e$.
Then, there is a trivial one-to-one correspondence between the $r$-arborescences in $D$ containing $e$ and the $r$-arborescences in $D'$.
Indeed, for any $r$-arborescence $F$ in $D$ containing $e$, the edge set in $D'$ corresponding to $F - e$ is an $r$-arborescence in $D'$.
Conversely, for any $r$-arborescence $F'$ in $D'$, an $r$-arborescence in $D$ is obtained by adding $e$ to the edge set in $D$ corresponding to $F'$.
This implies that $D'$ is also root-connected with respect to the same root $r$.

Let $c\colon A \to \mathbb{R}_{+}$ be a nonnegative cost function. For a family $\mathcal{F} \subseteq 2^{V}$ and an edge $a \in A$, define $\mathcal{F}_a = \{Z \in \mathcal{F} \colon a \text{ enters } Z\}$. A vector $y \in \mathbb{R}^{\mathcal{F}}$ is called \textit{$c$-feasible} if $y \ge 0$ and $\sum_{Z\in \cF_{a}} y(Z) \leq c(a)$ for every $a \in A$. When $\mathcal{F} = \{Z \colon \emptyset \neq Z \subseteq V-r\}$, a $c$-feasible vector $y$ is in fact a dual solution to the minimum-cost $r$-arborescence problem. An edge $a \in A$ is called \textit{$c$-tight} if $\sum_{Z\in \cF_{a}} y(Z) = c(a)$. Bock~\cite{bock1971algorithm} and Fulkerson~\cite{fulkerson1974packing} proved the following min-max formula.

\begin{thm}[Bock, Fulkerson]\label{Bock--Fulkerson}
    Let $D=(V,A)$ be a root-connected digraph with root $r$, and let $c\colon A\to\bR_+$ be a cost function. Then the minimum cost of an $r$-arborescence of $D$ equals
    \begin{equation*}
        \max \left\{ \sum_{Z \subseteq V-r} y(Z) \colon y \text{ is } c\text{-feasible} \right\}.
    \end{equation*}
    Moreover, an optimal dual solution $y^\star$ can be chosen such that the family $\{Z \colon y^\star(Z) > 0\}$ is laminar. For any optimal dual solution $y^\star$, an $r$-arborescence $B \in \mathcal{A}_D$ is minimum-cost if and only if
    \begin{enumerate}[label=(\roman*)] \itemsep0em
        \item every edge $a \in B$ is $y^\star$-tight, and
        \item $d^\mathrm{in}_{B}(Z) = 1$ for every nonempty set $Z \subseteq V - r$ with $y^\star(Z) > 0$.
    \end{enumerate}
\end{thm}

Based on this result, we first prove the following.

\begin{lem}\label{lem:arborescence}
    Let $D = (V, A)$ be a root-connected digraph with root $r$, and let $c\colon A \rightarrow \mathbb{R}_{+}$ be a cost function. Let $F$ be an $r$-arborescence of second-smallest cost with respect to $c$. Then there exist edges $e, f \in A$ such that
    \[
        c(F) = \min \bigl\{ c(F') \colon F' \in \mathcal{A}_D,\, \{e, f\}\subseteq F' \bigr\}.
    \]
\end{lem}

\begin{proof}
    Let $y$ be a fixed optimal $c$-feasible function. Since $F$ is not a minimum-cost $r$-arborescence, by \cref{Bock--Fulkerson} it either contains an edge that is not $y$-tight, or there exists a nonempty set $Z \subseteq V - r$ with $y(Z) > 0$ such that $F$ enters $Z$ more than once.
    
    In the first case, let $e$ be a non-$y$-tight edge of $F$ and let $f$ be any other edge of $F$. In the latter case, let $e,f$ be two edges of $F$ entering $Z$. In either case, by \cref{Bock--Fulkerson}, any $r$-arborescence containing both $e$ and $f$ cannot be of minimum cost. Since $F$ is a second-smallest $r$-arborescence and contains both $e$ and $f$ by definition, we obtain
    \[
        \min \bigl\{ c(F') \colon F' \in \mathcal{A}_D \bigr\} < \min \bigl\{ c(F') \colon F' \in \mathcal{A}_D,\, \{e,f\}\subseteq F' \bigr\} = c(F),
    \]
    which completes the proof.
\end{proof}

The next theorem is a counterpart of Theorem~\ref{thm:ESA2025} for arborescences.

\thmarbs*

\begin{proof}
    We proceed by induction over $\ell$. For $\ell=1$ the statement is trivial and \cref{lem:arborescence} proves the statement for $\ell=2$. Let $\ell\geq 3$ and assume as induction hypothesis that the statement holds for all $\ell' < \ell$ and let $F$ be an $\ell$-th smallest $r$-arborescence. Based on the argument in \cref{lem:arborescence} there exist $e, f \in F$ such that
    \[
    x_{1} < \bigl\{ c(F') \colon F' \in \mathcal{A}_D,\, \{e,f\}\subseteq F' \bigr\} \leq c(F).
    \]
    Note that $e = f$ is possible, and otherwise the heads of $e$ and $f$ are different nodes.

    Consider the graph $D'$ obtained from $D$ by removing all edges except $e$ and $f$ that have the same head as $e$ or $f$ and contracting the edges $e$ and $f$. 
    As seen above, there is a trivial one-to-one correspondence between the $r$-arborescences in $D$ containing $\{e, f\}$ and the $r$-arborescences in $D'$, and $D'$ is also root-connected as $F \setminus \{e, f\}$ forms an $r$-arborescence in $D'$. %, which is denoted by $F/\{e, f\}$.
    Moreover, $F \setminus \{e, f\}$ is an $\ell'$-th smallest $r$-arborescence for some $\ell'<\ell$.
    This is because otherwise there are at least $(\ell - 1)$ $r$-arborescences $F'_i \in \mathcal{A}_{D'}$ such that $c(F'_1) < c(F'_2) < \dots < c(F'_{\ell - 1}) < c(F \setminus \{e, f\})$, which implies $x_1 < c(F'_1 \cup \{e, f\}) < c(F'_2 \cup \{e, f\}) < \dots < c(F'_{\ell - 1} \cup \{e, f\}) < c(F)$, contradicting $c(F) = x_\ell$.
    
    By the induction hypothesis, there exists $P' \subseteq F \setminus \{e, f\} \in \mathcal{A}_{D'}$ with $|P'| \leq 2(\ell'-1)$ such that
    \[
    c(F\setminus \{e,f\}) = \min \bigl\{ c(F') \colon F' \in \mathcal{A}_{D'},\, P' \subseteq F' \bigr\}.
    \]
    This concludes that $P = P' \cup \{e, f\}$ satisfies $|P| \le 2(\ell' - 1) + 2 \le 2(\ell - 1)$ and 
    \[
    x_\ell = c(F) = \min \bigl\{ c(F') \colon F' \in \mathcal{A}_{D},\, P \subseteq F' \bigr\},
    \]
    which completes the proof.
\end{proof}

%%%%%%%%%%%%%%%%
\medskip
\paragraph{Acknowledgment.} 
The research received further support from the Lend\"ulet Programme of the Hungarian Academy of Sciences -- grant number LP2021-1/2021, from the Ministry of Innovation and Technology of Hungary from the National Research, Development and Innovation Fund -- grant numbers ADVANCED 150556, ADVANCED 153096, and ELTE TKP 2021-NKTA-62, and from the Dynasnet European Research Council Synergy project -- grant number ERC-2018-SYG 810115.
Yutaro Yamaguchi was supported by JSPS KAKENHI Grant Number JP25H01114, JST CRONOS Japan Grant Number JPMJCS24K2, and JST ASPIRE Japan Grant Number JPMJAP2520.

\bibliographystyle{abbrv}
\bibliography{matchings}

\end{document}